\documentclass[a4paper,10pt]{article}
\usepackage{amssymb}
\usepackage{amsmath}
\usepackage{tikz}
\usetikzlibrary{matrix,arrows}

\newcommand{\qed}{
  \ifmmode
   \eqno{\qedsymbol}
  \else
    \leavevmode\unskip\penalty9999 \hbox{}\nobreak\hfill\hbox{\qedsymbol}
  \fi
}
\newcommand{\qedsymbol}{\leavevmode\vrule height 1.2ex width 1.1ex depth -.1ex}
\newenvironment{proof}{\begin{trivlist}\item[\hskip
\labelsep{\bf Proof.\quad}]}
{\hfill\qed\rm\end{trivlist}}

\newtheorem{theorem}{Theorem}[section]

\newtheorem{proposition}[theorem]{Proposition}
\newtheorem{lemma}[theorem]{Lemma}
\newtheorem{definition}[theorem]{Definition}

\newtheorem{question}[theorem]{Question}

\mathchardef\emptyset="001F
 at 12 truept

\def\X{{\mathcal X}}

\def\SF{{\mathcal {SF}}}

\begin{document}

\title{A short note on strongly flat covers of acts over monoids}
\author{Alex Bailey and James Renshaw}
\date{September 2013}
\maketitle
\let\thefootnote\relax\footnote{{\\ 
University of Southampton, Southampton, SO17 1BJ, England\\
Email: alex.bailey@soton.ac.uk\\
j.h.renshaw@maths.soton.ac.uk}}

\begin{abstract}
\noindent Recently two different concepts of covers of acts over monoids have been studied. That based on coessential epimorphisms and that based on Enochs' definition of a flat cover of a module over a ring. Two recent papers have suggested that in the former case, strongly flat covers are not unique. We show that these examples are in fact false and so the question of uniqueness appears to still remain open. In the latter case, we re-present an example due to Kruml that demonstrates that, unlike the case for flat covers of modules, strongly flat covers of $S-$acts do not always exist.
\end{abstract}

\smallskip



\section{Introduction and Preliminaries}
Let $S$ be a monoid. By a {\em right $S-$act} we mean a non-empty set $X$ together with an action $X\times S\to X$ given by $(x,s)\mapsto xs$ such that for all $x\in X, s,t\in S, x1=x$ and $x(st) = (xs)t$.  We refer the reader to~\cite{howie-95} for basic results and terminology in semigroups and monoids and to \cite{ahsan-08} and \cite{kilp-00} for all undefined terms concerning acts over monoids.

Enochs' conjecture, that all modules over a unitary ring have a flat cover, was finally proven in 2001. In 2008, Mahmoudi and Renshaw \cite{renshaw-08} initiated a study of flat covers of acts over monoids. Their definition of cover proved to be different to that given by Enochs and in 2012, Bailey and Renshaw \cite{bailey-12a} initiated a study of Enochs' definition of cover. We give here both definitions but note that we use a slightly different terminology to that used in~\cite{renshaw-08}.

\begin{definition}
\rm Let $S$ be a monoid, $A$ an $S-$act and let $\X$ be a class of $S-$acts closed under isomorphisms. 
\begin{enumerate}
\item We shall say that an $S-$act $C$ together with an $S-$epimorphism $f : C \to A$ is a {\em coessential-cover} of $A$ if there is no proper subact $B$ of $C$ such that $f|_B$ is onto. If in addition $C\in\X$ then we shall call it an $\X-${\em coessential-cover}.

\item By an $\X$-{\em precover} of $A$ we mean an $S-$map $g: P\to A$ for some $P\in \X$ such that
for every $S-$map $g':P'\to A$, for $P'\in \X$, there exists an $S-$map $f:P'\to P$ with $g'=gf$.
$$
\begin{tikzpicture}[description/.style={fill=white,inner sep=2pt}]
\matrix (m) [matrix of math nodes, row sep=3em,
column sep=2.5em, text height=1.5ex, text depth=0.25ex]
{P & A\\ 
&P'\\};
\path[->,font=\scriptsize]
(m-2-2) edge node[auto,below left] {$f$} (m-1-1)
(m-2-2) edge node[auto, right] {$g'$} (m-1-2)
(m-1-1) edge[->] node[auto,above] {$g$} (m-1-2);
\end{tikzpicture}
$$
If in addition the $\X-$precover satisfies the condition that each $S-$map $f:P\to P$ with $gf=g$ is an isomorphism, then we shall call it an {\em $\X-$cover}.
\end{enumerate}
\end{definition}

It was shown in~\cite[Lemma 2.1]{bailey-12b} that if $\X$ is the class of projective (or free) $S-$acts then these two definitions coincide, whereas if $S$ is the infinite monogenic monoid  and $\X=\SF$ is the class of strongly flat $S-$acts then it easily follows from~\cite[Corollary 3.3]{renshaw-08} that the 1-element $S-$act does not have an $\SF-$coessential-cover but does have an $\SF-$cover by~\cite[Corollary 5.6]{bailey-12a}.

\smallskip

It is also easy to show that $\X-$covers, when they exist, are unique up to isomorphism, whereas this is not true, in general,  for $\X-$coessential-covers. However the question as to the uniqueness of $\SF-$coessential-covers has remained open, as apparently has the question of whether all $S-$acts have an $\SF-$cover.

\section{Uniqueness of strongly flat covers}
Recently Qiao and Wei have published an example of a monoid that they claim demonstrates that some $S-$acts do not have unique $\SF-$coessential-covers~\cite[Example 2.5]{qiao-12a}. However their result is false.

\smallskip

Let
$$
S=\langle x_0,x_1,\ldots \mid k\ge1, x_0x_k=x_kx_0=x_0, x_k^k=x_k^{k+1}; i,j\ge 1, x_ix_j=x_j^2\rangle\cup\{1\}
$$
and define a right $S-$congruence $\rho$ on $S$ by  $(s,t)\in\rho$ if and only if either $(s,t)\in\langle x_0\rangle$ or $s,t\in\langle x_1,x_2,\ldots\rangle\cup\{1\}$. For notational convenience, let us denote $1$ by $x_i^0$ for any $i\ge0$. 

Now let $R_i=\langle x_i\rangle\cup\{1\}$ and define a right congruence $\sigma_i$ on $S$ by $(s,t)\in\sigma_i$ if and only if there exists $p,q\in R_i$ with $ps=qt$. It is clear that if $i\ne j$ then $\sigma_i\ne\sigma_j$. Qiao and Wei claim that $S/\sigma_i$ and $S/\sigma_j$ are distinct $\SF-$coessential-covers for $S/\rho$. However by \cite[Lemma 2.4]{renshaw-08} we see that $S/\sigma_i\cong S/\sigma_j$ if and only if there exists $u\in S$ such that $\sigma_i=\{(s,t)\in S\times S \mid (us,ut)\in\sigma_j\}$. We show now that this is indeed the case.

\smallskip

First let us suppose without loss of generality that $j>i\ge0$. Suppose also that $(s,t)\in\sigma_i$ so that there exists $p,q\in R_i$ with $ps=qt$. Assume without loss of generality that
 $s\ne t$. We consider two cases:
\begin{enumerate}
\item Suppose that $i=0$. Then $\sigma_0=S\times S=\{(s,t)\in S\times S \mid (x_0s,x_0t)\in\sigma_j\}$.
\item Suppose that $i>0$.
\begin{enumerate}
\item If $s=1$. Then $t\in \langle x_i\rangle$ and so $x_j^jx_is=x_j^jx_it$.
\item If $s\in\langle x_k\rangle, k\ge 1$ then $t\in R_k$ (notice that if $t=1$ then $k=i$) and so $x_j^kx_is=x_j^kx_it$. 
\end{enumerate}
In both cases we deduce that $(s,t)\in\{(s,t)\in S\times S \mid (x_is,x_it)\in\sigma_j\}$.

Conversely, if $(x_is,x_it)\in\sigma_j$ then there exists $p',q'\in R_j$ with $p'x_is=q'x_it$. But $p'x_i, q'x_i\in\langle x_i\rangle$ and so $(s,t)\in\sigma_i$.
\end{enumerate}
Consequently $S/\sigma_i\cong S/\sigma_j$ and the $\SF-$coessential-covers are not distinct.

Notice in fact that by~\cite[Proposition 2.8]{ershad-11} $S/\rho$ does indeed have a unique $\SF-$coessential-cover since every element of the form $x_j^j$ is a right zero in $[1]_\rho$.

\medskip

There is a similar mistake to be found in \cite[Example 3.4]{qiao-12} where the same claim is made. Indeed the monoid in this example is finite whereas Ershad and Khosravi in~\cite{ershad-11} give an extensive list of monoids, which include the finite ones, where $\SF-$coessential-covers are unique, when they exist. A much simpler example with the same property would be the monoid $S=R^1$ where $R$ is any right zero semigroup. Given any $z \in R$, define $R_z=\{1,z\}$ and $(s,t) \in \sigma_z$ if and only if there exists $p,q \in R_z$ such that $ps=qt$. Then for every $z_1,z_2 \in R$, $\sigma_{z_1} \neq \sigma_{z_2}$ and $S/\sigma_{z_1}$ and $S/\sigma_{z_2}$ are both $\SF-$coessential-covers of the $1-$element $S-$act. However, $S/\sigma_{z_1}$ and $S/\sigma_{z_2}$ are both isomorphic.  Consequently the question of when $\SF-$coessential-covers are unique would appear still to be an open one.

\section{Existence of $\SF$-covers}

Enochs, Bican and El Bashir finally proved that all modules over a ring have flat covers in 2001. Similar results have subsequently been proved in a number of other categories and the obvious analogue for acts over monoids would be the existence of $\SF-$covers. However it has recently been brought to our attention that Kruml~\cite{kruml-08} has provided an example to show that $\SF-$covers do not always exist. The result below is essentially Kruml's, our contribution being to translate the proof from the language of varieties to the language of $S-$acts. We also pose a new question about the existence of $\SF-$covers.

\begin{proposition}[Cf. {\cite[Proposition 3.1]{kruml-08}}]
Let
$$T=\langle a_0,a_1,a_2 \dots \mid a_ia_j=a_{j+1}a_i \text{ for all }i \le j\rangle
$$
and let $S=T^1$, then the one element $S$-act $\Theta_S$ does not have an $\SF$-precover.
\end{proposition}

\begin{proof}
We first note that $S$ is left cancellative. In fact, every word $w \in T$ has a unique normal form $w=a_{\alpha(1)}\cdots a_{\alpha(n)}$ where $\alpha(i) \le \alpha(i+1)$ for all $1 \le i \le n-1$, and given any $a_{\alpha(n+1)},a_{\beta(n+1)}$, it is easy to see that $wa_{\alpha(n+1)}=wa_{\beta(n+1)}$ implies $\alpha(n+1)=\beta(n+1)$. Hence every $S-$endomorphism $h:S \to S$ is injective, as $h(s)=h(t)$ implies $h(1)s=h(1)t$.

\smallskip

Assume $\Theta_S$ does have an $\SF$-precover, then by \cite[Lemma 4.7]{bailey-12a}, $\SF$ contains a weakly terminal object, say $T$. By \cite[Theorem 5.3]{stenstrom-71}, let $(T,\alpha_i)_{i \in I}$ be the directed colimit of finitely generated free $S$-acts $(T_i,\phi_{i,j})_{i \in I}$. Let $X$ be any totally ordered set with $|X|>\max\{|I|,\aleph_0,|S|\}$ and let Fin$(X)$ denote the set of all finite subsets of $X$. We now define a direct system indexed over Fin$(X)$ partially ordered by inclusion, where every object $S_Y$ is isomorphic to $S$ and a map from an $n-1$ element subset $Y$ into an $n$ element subset $Y \cup \{z\}$ is defined to be the endomorphism $\lambda_{a_i}:S \to S$, $s \mapsto a_is$, where $i=|\{y \in Y \mid y<z\}|$. It follows from the presentation of $S$ that this is indeed a direct system.
Let $(F,\beta_Y)_{Y \in \text{Fin}(X)}$ be the directed colimit of this direct system, which by \cite[Proposition 5.2]{stenstrom-71}, is a strongly flat act. Therefore, there exists an $S$-map $t:F \to T$. Now for each singleton $\{x\} \in $ Fin$(X)$, $t\beta_{\{x\}}(1) \in T$ and so there exists some $i \in I$ and $x_i \in T_i$ such that $\alpha_i(x_i)=t\beta_{\{x\}}(1)$. Define the $S$-map $\theta_i:S_{\{x\}} \to T_i$, $s \mapsto x_is$ and then $t\beta_{\{x\}}=\alpha_i\theta_i$. So by the axiom of choice we can define a function $h:X \to Z, x \mapsto \left(i,x_i\right)$ where $Z:=\{(i,x) \in \{i\} \times T_i \mid i \in I\}$ and $|Z| \le \max\{|I|,\aleph_0,|S|\}$. Since $|X| > |Z|$, $h$ cannot be an injective function and so there exist $x \neq y \in X$ with $h(x)=h(y)$. Since $\theta_i$ is determined entirely by the image of $1$, we have that $t\beta_{\{x\}}=\alpha_i\theta_i=t\beta_{\{y\}}$. Without loss of generality, assume $x < y$ in $X$, then $\beta_{\{x,y\}}\lambda_{a_1}=\beta_{\{x\}}$ and $\beta_{\{x,y\}}\lambda_{a_0}=\beta_{\{y\}}$. Similarly there also exists $j \in I$, $\theta_j \in $ Hom$(S_{\{x,y\}},T_j)$ such that $t\beta_{\{x,y\}}=\alpha_j \theta_j$. Therefore we have
\begin{align*}
&\alpha_i \theta_i=t\beta_{\{x\}}=t\beta_{\{x,y\}}\lambda_{a_1}=\alpha_j\theta_j\lambda_{a_1} \\
\Rightarrow \quad &\alpha_i\left(\theta_i(1)\right)=\alpha_j\left(\theta_j\lambda_{a_1}(1)\right)
\end{align*}
and so by \cite[Theorem 2.2]{bailey-12a},there exists some $k\ge i, j$ such that $\phi_{i,k}\left(\theta_i(1)\right)=\phi_{j,k}\left(\theta_j\lambda_{a_1}(1)\right)$ which implies $\phi_{i,k}\theta_i=\phi_{j,k}\theta_j\lambda_{a_1}$. Similarly
\begin{align*}
&\alpha_i \theta_i=t\beta_{\{y\}}=t\beta_{\{x,y\}}\lambda_{a_0}=\alpha_j\theta_j\lambda_{a_0}=\alpha_k\phi^j_k\theta_j\lambda_{a_0} \\
\Rightarrow \quad &\alpha_i\left(\theta_i(1)\right)=\alpha_k\left(\phi^j_k\theta_j\lambda_{a_0}(1)\right)
\end{align*}
which again, implies there exists some $m \ge i,k$ such that $\phi_{i,m}\theta_i=\phi_{k,m}\phi_{j,k}\theta_j\lambda_{a_0}=\phi_{j,m}\theta_j\lambda_{a_0}$. Therefore
\[
\phi_{j,m}\theta_j\lambda_{a_1}=\phi_{k,m}\phi_{j,k}\theta_j\lambda_{a_1}=\phi_{k,m}\phi_{i,k}\theta_i=\phi_{i,m}\theta_i=\phi_{j,m}\theta_j\lambda_{a_0}.
\]
Since both $T_j$ and $T_m$ are finitely generated free $S$-acts, and $S_{\{x,y\}}$ is a cyclic $S$-act, it is clear that $\phi_{j,m}\theta_j$ is an endomorphism of $S$ and so a monomorphism. Therefore $\lambda_{a_0}=\lambda_{a_1}$ which implies $a_0=a_1$ which is a contradiction.
\end{proof}

It is still an open question to find necessary and sufficient conditions for the existence of $\SF-$covers for the category of acts over monoids.

\bigskip

Recall that a ring/monoid is called {\em right perfect} if every right module/act over it has a projective cover. Bass proved in 1960 that a ring is right perfect if and only if it satisfies $M_L$, the descending chain condition on principal left ideals \cite{bass-60}. It was shown in \cite{isbell-71} and \cite{fountain-76} that the case for monoids is different. A monoid is right perfect if and only if it satisfies $M_L$ and Condition $(A)$, every right $S-$act has the ascending chain condition on cyclic subacts. It was shown in \cite[Proposition 5.7]{bailey-12a} that a monoid $S$ satisfying Condition $(A)$ is a sufficient condition for every $S-$act to have an $\SF-$cover. The converse however is not true, the infinite monogenic monoid being a counterexample. It was also shown in \cite[Corollary 5.6]{bailey-12a} that $S$ being right cancellative is sufficient for every $S-$act to have an $\SF$-cover, and as we can see from the next Lemma, right cancellativity and $M_L$ implies Condition $(A)$.

\begin{lemma}
A right cancellative monoid with $M_L$ is a group.
\end{lemma}

\begin{proof}
Given any $s \in S$, consider the chain $Ss \supseteq Ss^2 \supseteq Ss^3 \supseteq \dots$, by $M_L$ there exists some $n \in \mathbb{N}$ such that $Ss^n=Ss^{n+1}$ and by right cancellativity this implies $S=Ss$.
\end{proof}

\smallskip

Since the only known counterexample to the existence of $\SF-$covers, as given above, does not have $M_L$, it seem natural to pose the following question

\begin{question}
Is it true that a monoid $S$ is right perfect if and only if it satisfies $M_L$ and every right $S-$act has an $\SF$-cover?
\end{question}

This clearly generalises the situation for rings where modules always have flat covers and a ring is perfect if and only if it has $M_L$. Clearly one way is obvious as a perfect monoid has Condition $(A)$ and so every $S-$act has an $\SF-$cover. The converse however is not clear. We would like to know if there exists a monoid $S$ satisfying $M_L$ but not Condition $(A)$ for which every $S-$act has an $\SF-$cover. Another class of monoids known to have $\SF-$covers are those monoids having weak finite geometric type (see \cite[Proposition 5.4]{bailey-12a}). This would seem to be a good place to look for a counterexample, although the main example of a monoid having weak finite geometric type, that is not right cancellative, is the Bicyclic monoid which does not have $M_L$.


\begin{thebibliography}{00}
\bibitem{ahsan-08} Ahsan, J., Zhongkui, L. {\it A Homological Approach to the Theory of Monoids}, Science Press, Beijing, (2008).

\bibitem{bailey-12a} Bailey, A.,  Renshaw, J., Covers of acts over monoids and pure epimorphisms, {\it Proc. Edinburgh Math. Soc.}, to appear, arXiv:1206.3095.

\bibitem{bailey-12b} Bailey, A., Renshaw, J. Covers of acts II, to appear in {\em Semigroup Forum}.

\bibitem{bass-60} Bass, H., Finitistic dimension and a homological generalization of semi-primary rings, {\it Trans. Amer. Math. Soc.} 95 (1960) 466--488.

\bibitem{ershad-11} Ershad, M., Khosravi, R., On the uniqueness of strongly flat covers of cyclic acts, {\it Turk J Math} 35 (2011) 437--442.

\bibitem{fountain-76} Fountain, J., Perfect semigroups, {\it Proc. Edinburgh Math. Soc} 20 (1976) 87--93.

\bibitem{howie-95} Howie, J.M. {\it Fundamentals of Semigroup Theory}, London Mathematical Society Monographs, (OUP, 1995).

\bibitem{isbell-71} Isbell, John R., Perfect monoids, {\it Semigroup Forum} 2 (1971) 95--118.

\bibitem{kilp-00} Kilp, M., Knauer, U., Mikhalev, A.V. {\it Monoids, Acts and Categories}, De Gruyter Expositions in Mathematics {\bf (29)}, (Walter de Gruyter, Berlin, New York, 2000).

\bibitem{kruml-08} Kruml, D., On flat covers in varieties, {\it Comment. Math. Univ. Carolin.} 49 (2008), 19--24.

\bibitem{qiao-12} Qiao, H., Wang, L., On flatness covers of cyclic acts over monoids, {\it Glasg. Math. J.} 54 (2012) 163--167.

\bibitem{qiao-12a} Qiao H., Wei C., On two open problems of Mahmoudi and Renshaw.
http://www.paper.edu.cn/index.php/default/en\verb!_!releasepaper/downPaper/201212-280 (2012).

\bibitem{renshaw-08} Mahmoudi, M., Renshaw, J., On covers of cyclic acts over monoids, {\it Semigroup Forum} (2008) 77, 325--338.

\bibitem{stenstrom-71} Stenstr{\"o}m, B., Flatness and localization over monoids, {\it Math. Nachr.} (1971) 48, 315--334.

\end{thebibliography}
\end{document}